\documentclass[12pt,reqno]{amsart}
\usepackage[all,2cell]{xy}

\usepackage{amssymb}

\newtheorem{theorem}{Theorem}[section]
\newtheorem{proposition}[theorem]{Proposition}
\newtheorem{lemma}[theorem]{Lemma}
\newtheorem{corollary}[theorem]{Corollary}

\theoremstyle{definition}
\newtheorem{definition}[theorem]{Definition}
\newtheorem{remark}[theorem]{Remark}

\numberwithin{equation}{section}

\DeclareMathOperator*{\Supp}{Supp}
\DeclareMathOperator*{\Vect}{Vect}
\DeclareMathOperator*{\Gal}{Gal}

\newcommand{\cO}{\mathcal{O}}
\newcommand{\scrF}{\mathcal{F}}
\newcommand{\scrE}{\mathcal{E}}
\newcommand{\scrV}{\mathcal{V}}
\newcommand{\scrK}{\mathcal{K}}
\newcommand{\scrC}{\mathcal{C}}
\newcommand{\scrf}{\mathfrak{f}}

\begin{document}

\title[Characterization of genuine ramification using formal orbifolds]{Characterization of
genuine ramification using formal orbifolds}

\author[I. Biswas]{Indranil Biswas}

\address{Department of Mathematics, Shiv Nadar University, NH91, Tehsil
Dadri, Greater Noida, Uttar Pradesh 201314, India}

\email{indranil.biswas@snu.edu.in, indranil29@gmail.com}

\author[M. Kumar]{Manish Kumar}

\address{Statistics and Mathematics Unit, Indian Statistical Institute,
Bangalore 560059, India}

\email{manish@isibang.ac.in}

\author[A.J. Parameswaran]{A. J. Parameswaran}

\address{School of Mathematics, Tata Institute of Fundamental
Research, Homi Bhabha Road, Bombay 400005, India}

\email{param@math.tifr.res.in}

\subjclass[2010]{14H30, 14G17, 14H60}

\keywords{Genuine ramification, branch data, formal orbifold, fundamental group}

\begin{abstract}
We give a characterization of genuinely ramified maps of formal orbifolds in the Tannakian framework. In particular we show that a 
morphism is genuinely ramified if and only if the pullback of every stable bundle remains stable in the orbifold category. We
also give some other
characterizations of genuine ramification. This generalizes the results of \cite{BKP-GR1} and \cite{BP-GR}. In fact, it is a positive 
characteristic analogue of results in \cite{GR-char0}.
\end{abstract}

\maketitle

\section{Introduction}

Let $f:Y\,\longrightarrow\, X$ be a finite generically smooth morphism of smooth projective connected curves over an algebraically closed field $k$. Then $f$ is
called genuinely ramified if there is no intermediate nontrivial \'etale cover of $X$. This notion
of genuine ramification admits several equivalent formulations. For instance $f$ is genuinely ramified if and only if
the maximal semistable subsheaf of $f_*\cO_Y$ is $\cO_X$, or if and only if the induced map of the \'etale fundamental groups is surjective,
or if and only if $f^*E$ is stable for every stable vector bundle $E$ on $X$, etcetera (see \cite{BP-GR}).

This notion of genuine ramification extends to the more general context of orbifolds. A formal orbifold curve is a pair $(X,\,P)$, where 
$X$ is a smooth projective curve and $P$ is a ``branch data'' on $X$ (the definition is recalled in Section \ref{se1}; see
\cite{KP} for more 
details). Since we will only be dealing with curves, for us a formal orbifold will mean a formal orbifold curve. A morphism of formal orbifolds $(Y,\,Q)\,\longrightarrow\, (X,\,P)$ is a finite generically smooth morphism of curves
$f\,:\,Y\,\longrightarrow\, X$ such that $Q(y)\,\supset\, P(f(y))$ for all closed points $y\,\in\, Y$. This morphism
of formal orbifolds is \'etale if the equality $Q(y)\,=\, P(f(y))$ holds 
for all $y\,\in\, Y$. A morphism of formal orbifolds is called genuinely ramified if there is no intermediate nontrivial \'etale cover
(see Definition \ref{3.1}).

A ``geometric'' formal orbifold $(X,\,P)$ is one for which there exists a Galois \'etale cover of formal 
orbifolds $g\,:\,(W,\,O)\,\longrightarrow\, (X,\,P)$, where $O$ denotes the trivial branch data. In \cite{KP} 
a vector bundle on such a pair $(X,\,P)$ was defined to be a $G$-equivariant vector bundle on $W$, where $G$ 
is the Galois group for the map $g$. The notions of degree, slope, semistability etcetera were also defined in 
\cite{KP}. When the characteristic of the base field $k$ is zero, a branch data on $X$ is the same as 
assigning integers greater than 1 at finitely many closed points of $X$. This in turn is equivalent to giving 
$X$ an orbifold structure. When the characteristic of the base field $k$ is zero, the equivariant bundles are 
also called orbifold bundles and they are equivalent to (see \cite{Bi}) the parabolic bundles introduced by 
Mehta and Seshadri (\cite{MS}, see also \cite{MY}). In positive characteristic this result was proved in 
\cite{KM} where the notion of parabolic bundles is different and it is a slight variant of Nori's definition 
(given in \cite{Nori}).

The following results were proved in \cite{GR-char0}. However their formulation in \cite{GR-char0} involved 
the terminology of orbifolds.

\begin{proposition}\label{prop-i}
Let $f\,:\,Y\, \longrightarrow \,X$ be a morphism of irreducible smooth curves over an algebraically closed field of characteristic zero.
Let $X$ be equipped with a branch data $P$. Then $f_*{\mathcal O}_Y$ is a semistable ``parabolic bundle'' of degree zero belonging to 
$\Vect(X,\,B_f)$ and so it is semistable in $\Vect(X,\,PB_f)$. Then the maximal subbundle of $f_*{\mathcal O}_Y$ belonging to $\Vect(X,\,P)$ is 
a sheaf of subalgebras in $f_*{\mathcal O}_Y$ under parabolic tensor product. Moreover the corresponding spectrum defines the maximal
cover $g\,:\,Z\,\longrightarrow\, X$ such that $(Z,\,g^*P)\,\longrightarrow\, (X,\,P)$ is \'etale and $f$ dominates $g$.
\end{proposition}

\begin{theorem}\label{thm-i}
Let $f\,:\,Y\, \longrightarrow \,X$ be a morphism of irreducible smooth curves over an algebraically closed field of characteristic
zero, and let $P$ be a branch data on $X$. Then the following statements are equivalent:
\begin{enumerate}
\item The induced map $\pi(f)\,:\,\pi_1(Y,\,f^*P)\,\longrightarrow\, \pi_1(X,\,P)$ is surjective.

\item Every stable parabolic bundle in $\Vect(X,\,P)$ pulls back to a stable
parabolic vector bundle in $\Vect(Y,\,f^*P)$.

\item The maximal subbundle of $f_*{\mathcal O}_Y$ of degree zero belonging to $\Vect(X,\,P)$ is ${\mathcal O}_X$.
\end{enumerate}
\end{theorem}

We extend Proposition \ref{prop-i} and Theorem \ref{thm-i} to the fields of positive characteristic. More precisely, in Theorem
\ref{stablepullback} we show that the above (1) and (2) are equivalent. This improves the main result of \cite{BKP-GR1}. In Theorem \ref{5.2}, (1) and (3) are shown to be equivalent under 
the additional hypothesis that $f$ is Galois. In Proposition \ref{main}, a characterization of genuinely ramified morphisms is given in
the Tannakian framework.

\section{Vector bundles over orbifolds}\label{se1}

Let $X$ be a smooth projective connected curve defined over an algebraically closed field $k$. We briefly recall some basic
definitions from \cite[Section 2]{KP}. For a closed point $x\,\in\, X$, let $\scrK_{X,x}$ denote the fraction field of $\widehat \cO_{X,x}$. 
A \textit{branch data} $P$ on $X$ assigns to each closed point $x\,\in\, X$ a finite Galois extension $P(x)$ of $\scrK_{X,x}$,
satisfying the condition that the extension is trivial for all but finitely many points. The support $\Supp(P)$
of the branch data $P$ is the finite subset of closed points where the field extension is actually nontrivial.

Given any two branch data $P_1$ and $P_2$ on $X$, we say that $P_1\,\leq\, P_2$ if
$P_1(x)\,\subset\, P_2(x)$ for all
closed points $x\,\in \,X$. We can define their intersection $P_1\bigcap P_2$ by $(P_1\bigcap P_2)(x)\,:=\,P_1(x)\bigcap 
P_2(x)$ for all closed points $x\,\in\, X$. Here the intersection is taken in a fixed algebraic
closure of $\scrK_{X,x}$. Note that $$\Supp (P_1\cap P_2) \,\,\subseteq\,\, \Supp (P_1)\cap \Supp ( P_2).$$
Similarly we also define their compositum $P_1P_2$ by $$(P_1P_2)(x)\,:=\,P_1(x)\cdot P_2(x)$$
for all closed points $x\,\in\, X$. Note that 
$\Supp (P_1P_2) \,=\, \Supp (P_1) \cup \Supp ( P_2)$. Also we have
$$P_1\cap P_2 \,\,\,\leq\,\,\, P_i\, \,\,\leq\,\,\, P_1P_2$$ for $i\,=\, 1,\, 2$.

The trivial branch data is the one where all the field extensions are trivial; the trivial branch data is denoted by $O$. Let
$f\,:\,Y\,\longrightarrow\, X$ be a finite generically
smooth map and $P$ a branch data on $X$. Then there is a natural branch data $f^*P$ on $Y$ constructed as follows: 
For any closed point $y\in Y$, the field $f^*P(y)$ is the compositum $P(f(y))\scrK_{Y,y}$.

Let $f\,:\,Y\,\longrightarrow\, X$ be a finite generically smooth map. For a closed point $x\,\in\, X$, let $B_f(x)$ be the 
compositum of the Galois closures of the field extensions $\{\scrK_{Y,y}\}_{y\in f^{-1}(x)}$ of $\scrK_{X,x}$. So we have a branch
data $B_f$ that assigns $B_f(x)$ to any $x\, \in\, X$. The support of $B_f$ is evidently the subset over which $f$ is ramified. Note 
that if $f$ is also Galois, then $B_f(x)\,=\,\scrK_{Y,y}$ for any $y\,\in\, Y$ with $f(y)\,=\,x$.

A formal orbifold curve is a pair $(X,\,P)$, where $X$ is a smooth projective curve with $P$ being a branch data on
$X$. A \textit{morphism} of orbifold curves $$f\,:\,(Y,\,Q)\,\longrightarrow\, (X,\,P)$$ is a finite generically smooth morphism
$f\,:\,Y\,\longrightarrow\, X$ such that $Q(y)\,\supset\, P(f(y))$ for every closed point $y\,\in\, Y$. 
The map $f$ is said to be \textit{\'etale} if $Q(y)\,=\,P(f(y))$ for all closed points $y\,\in\, Y$.

For any finite generically smooth morphism $f\,:\,Y\, \longrightarrow \,X$,
$$f\,:\,(Y,\,f^*B_f)\, \longrightarrow \, (X,\,B_f)$$ is an \'etale cover (\cite[Lemma 2.12]{KP}).
We note that a
Galois morphism $f\,:\, (Y,\, O)\,\longrightarrow\, (X,\,P)$ is \'etale if and only if $P\,=\, B_f$. Also, we have $f^*B_f\,=\,O$
for any Galois \'etale morphism $f$.

A branch data $P$ on $X$ is said to be \textit{geometric} if there exists an \'etale Galois covering
map $f\,:\, (Y,\, O)\,\longrightarrow\, (X,\,P)$ of formal orbifolds. If $P$ is a
geometric branch data on $X$, then $(X,\,P)$ is called a \textit{geometric formal orbifold}.

Take a geometric formal orbifold $(X,\,P)$. Fix a Galois \'etale covering
$$f\,:\,(W,\, O)\,\longrightarrow\, (X,\,P)\, ;$$ 
the Galois group of $f$ will be denoted by $G$. An object $\scrV$ in the category $\Vect(X,\,P)$ is a
$G$-equivariant vector bundle $V$ on $W$, while morphisms in $\Vect(X,\,P)$ are the equivariant morphisms of
$G$-equivariant vector bundles on $W$. It should be clarified that the category
$\Vect(X,\,P)$ is actually independent of the choice of $f$ (see \cite[Proposition 3.6]{KP}).

Let
$$
\deg(\scrV)\, :=\, \frac{\deg(V)}{\deg(f)}
$$
be the degree of the object $\scrV$ of the category $\Vect(X,\,P)$. Also define the slope
$$\mu(\scrV) \,:=\, \frac{\deg(\scrV)}{{\rm rank}(V)}.$$
The object $\scrV$ is called \textit{stable} (respectively, \textit{semistable}) if 
$\mu(\scrF) \, <\, \mu(\scrV)$ (respectively, $\mu(\scrF) \, \leq\, \mu(\scrV)$) for all nonzero subobjects $\scrF\, \subset\, \scrV$ of
smaller rank. A semistable vector bundle $\scrV$ is called \textit{strongly semistable} if ${\scrV}^{\otimes j}$ is semistable for
all $j\, \geq\, 1$.

Recall from \cite{Nori} that a vector bundle $V$ is called finite if $p(V)\,=\,q(V)$ for two distinct polynomials 
$p$ and $q$ with nonnegative integer coefficients. For $a$ and $n$ nonnegative integers, $aV^n$ means
$\bigoplus_{i=1}^a V^{\otimes n}$. Any vector bundle of degree zero which is isomorphic to a subbundle of a
quotient bundle of degree zero of a finite bundle is called an essentially finite bundle. Let $\Vect^f(X,\,P)$ (respectively, $\Vect^{ss}(X,\,P)$) 
denote the full subcategory of $\Vect(X,\,P)$ consisting of essentially finite (respectively, strongly 
semistable of degree zero) vector bundles. Also, $\Vect^{et}(X,\,P)$ denotes the full subcategory of 
$\Vect(X,\,P)$ consisting of \'etale trivial vector bundles.

Let $x\,\in\, X$ be a closed point outside the support of $P$. Note that each of
$\Vect^{ss}(X,\,P)$, $\Vect^f(X,\,P)$ and $\Vect^{et}(X,\,P)$, 
equipped with the natural fiber functor associated to $x$, is a neutral Tannakian category; their Tannaka duals are denoted by
$\pi^S(X,\,P)$, $\pi^N(X,\,P)$ and $\pi_1^{et}(X,\,P)$ respectively.

In \cite[Section 3.2]{KP}, for an \'etale morphism $f\,:\,(X_1,\,P_1)\,\longrightarrow\, (X_2,\,P_2)$ a pushforward functor
$\Vect(X_1,\,P_1)\,\longrightarrow\, \Vect(X_2,\,P_2)$ was defined. To explain this construction,
let $(Y_2,\,O)\,\longrightarrow\, (X_2,\,P_2)$ be a Galois \'etale covering map; denote by $\Gamma$ the
Galois group of this map. Let $Y_1$ be the normalization of the fiber product $\widetilde{X_1\times _{X_2} Y_2}$ over $X_2$.
Then the natural projection $\scrf\,:\,Y_1\longrightarrow \, Y_2$ is an \'etale cover. Also, $\scrf$ is a $\Gamma$-equivariant morphism.
Moreover, the natural map $(Y_1,\,O) \,\longrightarrow\, (X_1,\,P_1)$ is an \'etale Galois covering
with Galois group $\Gamma$. Hence an object $\mathcal V$ of $\Vect(X_1,\,P_1)$ is a $\Gamma$-equivariant vector bundle $V$ on $Y_1$. The
direct image $\scrf_*V$ is a $\Gamma$-equivariant vector bundle on $Y_2$, and hence it is an object of $\Vect(X_2,\,P_2)$.
We define this object of $\Vect(X_2,\,P_2)$ as the \emph{direct image} of $\mathcal V$; this direct image will be
denoted by
\begin{equation}\label{d}
\widehat{f}_*{\mathcal V}.
\end{equation}
It should be mentioned that the object $\widehat{f}_*{\mathcal V}$ of $\Vect(X_2,\, P_2)$
is actually independent of the choice of the Galois \'etale covering map $(Y_2,\,O)\,\longrightarrow\, (X_2,\,P_2)$.

\begin{lemma}\label{semistability}
Let $P_1\,\leq\, P_2$ be two branch data. Then there is a natural fully faithful functor $\Vect(X,\,P_1)\,\longrightarrow\, \Vect(X,\,P_2)$.
Moreover, a vector bundle $\scrV$ in $\Vect(X,\, P_1)$ is semistable if and only if it is semistable as a vector bundle in $\Vect(X,\,P_2)$.
\end{lemma}

\begin{proof}
The first statement is proved in \cite[Theorem 3.7]{KP} (also see \cite[Theorem 2.5]{tamenori}).
The second statement follows from \cite[Lemma 3.10]{KP}.
\end{proof}

\begin{definition}\label{2.2}
Let $P_1$ and $P_2$ be two geometric branch data on $X$. An object $$\scrE\,\in\, \Vect(X,\,P_1)$$ is said to be from
$\Vect(X,\,P_2)$ if there exists an object $\scrE'\,\in\, \Vect(X,\,P_2)$ such that the images of $\scrE$ and $\scrE'$ in
$\Vect(X,\, P_1P_2)$ --- under the functors 
$\Vect(X,\,P_1)\, \longrightarrow\, \Vect(X,\,P_1P_2)$ and $\Vect(X,\,P_2)\, \longrightarrow\, \Vect(X,\,P_1P_2)$
in Lemma \ref{semistability} --- are isomorphic.
\end{definition}

\begin{lemma}\label{directimage}
Let $f\,:\,Y\, \longrightarrow \,X$ be a finite generically smooth morphism of smooth projective curves.
Consider $\cO_Y$ as an object of $\Vect(Y,\,f^*B_f)$. Then $\widehat{f}_*{\mathcal O}_Y$ (see \eqref{d})
is a semistable vector bundle in $\Vect(X,\,B_f)$ of degree zero. 
\end{lemma}

\begin{proof}
When the characteristic of $k$ is $0$, this is proved in \cite{P}. 

The morphism $f$ is actually an \'etale morphism of orbifolds $f\,:\,(Y,\,f^*B_f)\,\longrightarrow\, (X,\,B_f)$. Let
\begin{equation}\label{gc}
X_e\,\longrightarrow\, Y\,\longrightarrow\, X
\end{equation}
be the morphism of smooth projective curves 
constructed by setting $k(X_e)$ to be the Galois closure of $k(Y)/k(X)$. Both the maps in $$(X_e,\,O)\,\longrightarrow\,
(Y,\,f^*B_f)\,\longrightarrow\, (X,\,B_f)$$ are \'etale. Denote $G\,=\,\Gal(k(X_e)/k(X))$.
Note that $\cO_Y$ is in $\Vect(Y,\,O)$ and hence it is also in $\Vect(Y,\,f^*B_f)$.

Let $Y_e\,=\, \widetilde{Y\times_X X_e}$ be the normalization of the fiber product $Y\times_X X_e$
over $X$. The action of $G$ on $X_e$ and
the trivial action of $G$ on $Y$ together produce an action of $G$ on $Y_e$.
The projection $$h\,:\,(Y_e,\,O)\,\longrightarrow\, (Y,\,
f^*B_f)$$ is an \'etale Galois covering with Galois group $G$. Clearly,
$\cO_{Y_e}\,=\,h^*\cO_Y$ is a $G$-equivariant line bundle. Hence $\cO_Y$ --- as an object of $\Vect(Y,\,f^*B_f)$ --- is
the $G$-equivariant bundle $\cO_{Y_e}$. Let $\scrf\,:\,Y_e\,=\, \widetilde{Y\times_X X_e} \, \longrightarrow\, X_e$
be the natural projection. By \cite[Section 3.2]{KP}, the direct image $\widehat{f}_*\cO_Y$ (see \eqref{d}) is the
$G$-equivariant vector bundle $\scrf_*\cO_{Y_e}$ on $X_e$. Since $\scrf$ is \'etale, the direct image ${\scrf}_*
{\mathcal O}_{Y_e}$ is semistable of degree zero \cite[p.~12825, Lemma 2.3]{BP-GR}. This proves the lemma.
\end{proof}

\begin{lemma}\label{maximalsubbundle}
Let $P_1\,\le\, P_2$ be two branch data on $X$. Let $\scrV\,\in \,\Vect(X,\,P_2)$ be a semistable vector bundle admitting a subbundle $\scrV\, \supset\, {\scrV}'\,\in\, 
\Vect(X,\,P_1)$ (see Definition \ref{2.2}) such that $\mu(\scrV)\,=\, \mu({\scrV}')$. Then there is a unique maximal
semistable subbundle $\scrV_1\,\subset\, \scrV$ such that
\begin{enumerate}
\item $\mu(\scrV_1)\,=\, \mu(\scrV)$, and

\item $\scrV_1\, \in\, \Vect(X,\,P_1)$.
\end{enumerate}
\end{lemma}

\begin{proof}
Since $\scrV$ is semistable, and the slope of the above vector bundle ${\scrV}'$ coincides with that of $\scrV$, it follows that
${\scrV}'$ is also semistable. If ${\scrV}'$ and ${\scrV}''$ are two subbundles of $\scrV$ lying in $\Vect(X,\,P_1)$ such that
$\mu({\scrV}'')\,=\, \mu(\scrV)\,=\, \mu({\scrV}')$, then their sum 
${\scrV}'+{\scrV}''\, \subset\, \scrV$ is again a subbundle with
$\mu({\scrV}'+{\scrV}'')\, =\, \mu(\scrV)$. Indeed, ${\scrV}'+{\scrV}''$ is a quotient of ${\scrV}'\oplus {\scrV}''$,
so $\mu({\scrV}'+{\scrV}'')\, \geq\, \mu({\scrV}')\,=\, \mu({\scrV}'')$, on the other hand, ${\scrV}'+{\scrV}''$ is a subsheaf of
${\scrV}$, so $\mu({\scrV}'+{\scrV}'')\, \leq\, \mu({\scrV})$. We also have ${\scrV}'+{\scrV}''\, \in \, \Vect(X,\,P_1)$.
This proves the existence and uniqueness of a maximal subbundle $\scrV_1\,\subset\, \scrV$ as in the lemma.
\end{proof}

\section{Pullback of stable bundles}

Let $f\,:\,Y\,\longrightarrow\, X$ be a finite generically smooth morphism between smooth connected projective 
curves. Let $P$ be a geometric branch data on $X$.

\begin{definition}\label{3.1}
We say $f\,:\,(Y,\,f^*P)\,\longrightarrow\, (X,\,P)$ to be a genuinely ramified map of formal orbifolds if there is no
intermediate cover $$(Y,\,f^*P)\,\longrightarrow\, (Z,\,Q)\, \longrightarrow\, (X,\,P)$$ where
$(Z,\,Q)\,\longrightarrow\, (X,\,P)$ is a nontrivial \'etale cover of formal orbifold curves. 
\end{definition}

\begin{lemma} \label{4-main}
Let $f\,:\,(Y,\,f^*P)\,\longrightarrow\,(X,\,P)$ be genuinely ramified. Let $(W,\,O)\,\longrightarrow\,(X,\,P)$ be an \'etale Galois
cover with Galois group $\Gamma$. Let
$$g\,:\,Z\, :=\, \widetilde{W\times_X Y}\,\longrightarrow\, W$$ be the normalization
of the fiber product $W\times_X Y$. Then $g$ is a genuinely
ramified morphism. Also the morphism $(Z,\,O)\,\longrightarrow\, (Y,\,f^*P)$ is a Galois \'etale cover
with Galois group $\Gamma$. 
\end{lemma}

\begin{proof}
If $L\,=\,k(W)\cap k(Y)\varsupsetneq k(X)$, then the normalization
$$f'\,:\,Y'\,\longrightarrow \, X$$ 
of $X$ in $L$ is of degree at least two. Note that $f'$ is an essentially \'etale cover of $(X,\,P)$
(see \cite[Definition 2.6(3)]{KP}) because $f'$ is dominated
by $W\,\longrightarrow \, X$. Hence by \cite[Lemma 2.12]{KP} $$f'\,:\,(Y',\,f'^*P)\,\longrightarrow \, (X,\,P)$$ is an \'etale cover,
which contradicts the given condition that $f\,:\,(Y,\,f^*P)\,\longrightarrow \, (X,\,P)$ is a genuinely ramified morphism. So
$k(W)\cap k(Y)\,=\,k(X)$.

Moreover, since $k(W)/k(X)$ is Galois, it follows that $k(W)$ and $k(Y)$ are linearly disjoint over $k(X)$. 
Hence the morphism $g\,:\,Z\,\longrightarrow \, W$ in the statement of the lemma is a finite generically 
smooth morphism of connected nonsingular curves, and $Z\,\longrightarrow \, Y$ is a Galois cover with Galois 
group $\Gamma$. Furthermore, since $(W,\,O)\,\longrightarrow \, (X,\,P)$ is \'etale, the pullback 
$(Z,\,O)\,\longrightarrow \, (Y,\,f^*P)$ is also \'etale by \cite[Proposition 2.14 and Proposition 2.16]{KP}.
 
Suppose that $g$ is not genuinely ramified. Let $h\,:\,W'\,\longrightarrow \, W$ be the maximal \'etale cover 
of $W$ dominated by $g$. For $\gamma\,\in \,\Gamma$, the automorphism of $W$ given by $\gamma$ will also be 
denoted by $\gamma$. The pullback of $h\,:\,W'\,\longrightarrow \, W$ by $\gamma$ is again \'etale. Since $W'$ 
is maximal \'etale, the pullback of $h\,:\,W'\,\longrightarrow \, W$ is again $h$. Hence $\Gamma$ acts on 
$W'$. Let $f''\,:\,Y''\,\longrightarrow \, X$ be the normalization of $X$ in $k(W')^{\Gamma}$. Then 
$Y''\,\longrightarrow \, X$ is dominated by $f\,:\,Y\,\longrightarrow \, X$. The following diagram summarizes 
the situation:
\begin{equation}\label{d1}
\xymatrix{
Z \ar[r]^{\Gamma}\ar[d]\ar@/_1.5pc/[dd]_g & Y\ar@/^1.5pc/[dd]^f\ar[d]\\
W' \ar[d]\ar[r] & Y''\ar[d]\\
W \ar[r]^{\Gamma} & X\\
}
\end{equation}
 
Also $(W',\,O)\,\longrightarrow \, (X,\,P)$ being the composition of two \'etale maps is also \'etale. Hence $Y''\,\longrightarrow \, X$
is essentially \'etale for $(X,\,P)$. So again by \cite[Lemma 2.12]{KP} $$f''\,:\,(Y'',\,f''^*P)\,\longrightarrow \, (X,\,P)$$
is \'etale, which gives a contradiction. Hence $g$ is genuinely ramified. This completes the proof.
\end{proof}

\begin{theorem}\label{stablepullback}
Let $f\,:\,(Y,\,f^*P)\,\longrightarrow \, (X,\,P)$ be a morphism of formal orbifolds. Then $f$ is genuinely ramified
if and only if the pullback of every stable object in $\Vect(X,\,P)$ is stable in $\Vect(Y,\,f^*P)$.
\end{theorem}

\begin{proof}
Let $a\,:\,(W,\,O)\,\longrightarrow \, (X,\,P)$ be a Galois \'etale covering with Galois group $\Gamma$. A stable object
of $\Vect(X,\,P)$ is a $\Gamma$-stable vector bundle $E$ on $W$. Let $$g\,:\,Z\, :=\,
\widetilde{W\times_X Y} \,\longrightarrow \, W$$ be the
normalized pullback of $f$ (see Lemma \ref{4-main}). By Lemma \ref{4-main}, the morphism $g\,:\,Z\,\longrightarrow \, Y$ is genuinely
ramified. Also note that $g$ is a $\Gamma$-equivariant morphism. By \cite[Proposition 4.2]{BKP-GR1}, the
pullback $g^*E$ is a
$\Gamma$-stable vector bundle on $Z$. The morphism $g\,:\,(Z,\,O)\,\longrightarrow \, (Y,\,f^*P)$ is Galois \'etale with Galois
group $\Gamma$. Hence $g^*E$ is a stable object in $\Vect(Y,\, f^*P)$.

If $f$ is not genuinely ramified, there is a nontrivial \'etale cover of formal orbifold curves
$$h\, :\, (Z,\,Q)\,\longrightarrow\, (X,\,P)$$
such that $f$ factors as 
$$(Y,\,f^*P)\, \stackrel{\phi}{\longrightarrow}\, (Z,\,Q)\, \stackrel{h}{\longrightarrow}\, (X,\,P).$$
Take any line bundle $L$ on $(Z,\, Q)$ of degree one. Then $\widehat{h}_*L$ is a stable vector bundle on
$(X,\, P)$. But $h^*\widehat{h}_*L$ is not stable because $L$ is a quotient of it. Since
$h^*\widehat{h}_*L$ is not stable, we conclude that $\phi^* h^*\widehat{h}_*L\,=\, f^*\widehat{h}_*L$ is not stable.
\end{proof}

\section{Tannakian characterization}\label{TC}

Let $f\,:\,Y\,\longrightarrow\, X$ be a finite generically smooth morphism between smooth connected projective curves.
As before, the branch data on $X$ given by $f$ will be denoted by $B_f$. As in \eqref{gc}, let $b\,:\, X_e
\,\longrightarrow\, X$ be the Galois closure of $f$ with Galois group $G$. Consider the normalization
\begin{equation}\label{e1}
\scrf\,\,:\,\, Y_e\,\,:=\,\, \widetilde{Y\times_X X_e}\,\, \longrightarrow\,\, X_e
\end{equation}
of the fiber product $Y\times_X X_e$. Note that $Y_e$ is $\deg(f)$-copies of $X_e$. As noted in the proof of
Lemma \ref{directimage}, the group $G$ acts on $Y_e$. The direct image
\begin{equation}\label{de}
E\, :=\, \scrf_*\cO_{Y_e},
\end{equation}
where $\scrf$ is the map in \eqref{e1}, is a $G$-equivariant vector bundle on $X_e$, because $\cO_{Y_e}$ is a
$G$-equivariant vector bundle and the projection $\scrf$ is a $G$-equivariant morphism. We have the diagram:
\begin{equation}\label{cd}
\xymatrix{
& & Y_e\ar[rdd]_a^G\ar[ld]_{\scrf} \\
& X_e\ar@{-->}[rrd]\ar[rdd]_b^G \\
& & & Y\ar[ld]_f\\
& & X
}
\end{equation}

We saw in the proof of Lemma \ref{directimage} that $E$ is $\widehat{f}_*\cO_Y\, \in\, \Vect(X,\,B_f)$.

\begin{remark}
If the map $f$ is Galois, then $Y\,=\,X_e$. So $Y_e\,=\, G\times X_e$, where $G\,=\,\text{Gal}(f)$.
Hence in that case we have $$E\,=\,\scrf_*\cO_{Y_e}\,=\,k[G]\otimes_k \cO_{X_e}.$$
\end{remark}

Let $\scrC(f)$ be the neutral Tannakian subcategory of $\Vect^{ss}(X,\,B_f)$ defined by the full subcategory 
generated by $E$ in \eqref{de}.

Note that $\cO_{X_e}\, \in\, \Vect^{ss}(X,\,B_f)$. Let $\scrC(b)$ denote the full 
neutral Tannakian subcategory of $\Vect^{ss}(X,\,B_f)$ generated by $k[G]\otimes_k\cO_{X_e}$.

\begin{proposition}\label{prop-n}
The equality $\scrC(f)\,=\,\scrC(b)$ holds. Hence the Tannaka dual of $\scrC(f)$ is the Galois group $G$.
\end{proposition}

\begin{proof}
Let $\widetilde{f}\, :\, X_e\, \longrightarrow\, Y$ be the map in \eqref{cd}.
The normalization of the fiber product $X_e\times_X X_e$ will be denoted by $M$. Let $\varphi\,:\, M\, \longrightarrow\, X_e$
be the projection to the second factor. The map
$$
\widetilde{f}\times {\rm Id}\,\, :\,\, X_e\times_X X_e\, \longrightarrow\, Y\times_X X_e,\,\ \
(x_1,\, x_2)\, \longmapsto\, (\widetilde{f}(x_1), \, x_2)
$$
produces a map $g\, :\, M\, \longrightarrow\, Y_e$. This map $g$ satisfies the equation
$$
\scrf\circ g\,=\, \varphi\, ,
$$
where $\scrf$ is the map in \eqref{e1}. This implies that 
\begin{equation}\label{z1}
E \, \subset\, \varphi_* {\mathcal O}_M
\end{equation}
(see \eqref{de}). But $\varphi_* {\mathcal O}_M\, =\, k[G]\otimes_k\cO_{X_e}$ because the
map $b$ in \eqref{cd} is Galois with Galois group $G$. So from \eqref{z1} it follows that
\begin{equation}\label{z2}
\scrf_*\cO_{Y_e} \, \subset\, k[G]\otimes_k\cO_{X_e}\, .
\end{equation}
{}From \eqref{z2} it follows immediately that $\scrC(f)$ is a full subcategory of $\scrC(b)$.

Consider the neutral Tannakian category $\text{Rep}(G)$ defined by all algebraic representations of $G$
in finite dimensional $k$--vector spaces. Consider the subgroup $$H\, :=\,\text{Gal}(a) \, \subset\, G$$
which is the Galois group of the morphism $a\, :\, Y_e\, \longrightarrow\, Y$ in \eqref{cd}. The left-translation
action of $G$ on $G/H$ makes $k[G/H]\, \in\, \text{Rep}(G)$. Let $\mathcal{C}(G/H)$ be the full neutral Tannakian
subcategory of $\text{Rep}(G)$ generated by $k[G/H]$. Since $b\, :\, X_e\, \longrightarrow\, X$ is the
Galois closure of $f$, it follows that 
$$
\mathcal{C}(G/H)\,=\, \text{Rep}(G)\, .
$$
{}From this it follows that the subcategory $\scrC(f)$ of $\scrC(b)$ actually coincides with $\scrC(b)$.
\end{proof}

In the set-up of Proposition \ref{prop-n}, let $P$ be a geometric branch data on $X$. Let $\scrC_P(f)$ be the full neutral
Tannakian subcategory of $\scrC(f)$ consisting of objects which are from $\Vect(X,\,P)$ (in the sense of \eqref{2.2}). Let $A$
be the Tannaka dual of $\scrC_P(f)$. Proposition \ref{prop-n} says that the Tannaka dual of $\scrC(f)$ is $G$. So we have
a natural epimorphism
\begin{equation}\label{alpha}
\alpha\,:\,G\,\longrightarrow\, A.
\end{equation}

\begin{lemma}\label{lem-n}
Let $H'\, \subset\, G$ be the kernel of the homomorphism $\alpha$ in \eqref{alpha}.
Let $\phi\,:\,Y'\,\longrightarrow\, X$ be the normalization of $X$ in $k(X_e)^{H'}$,
so $Y'\,=\, X_e/H'$. Then $\phi\,:\, (Y',\,\phi^*P)\,\longrightarrow\, (X,\,P)$ is the unique maximal \'etale
cover of $(X,\, P)$ dominated by $X_e$.
\end{lemma}

\begin{proof}
Let $Q$ denote the branch data on $Y'$ given by the quotient map $$X_e\, \longrightarrow\, X_e/H' \,=\, Y'.$$
The category $\scrC(b)$ in Proposition \ref{prop-n} contains only \'etale trivial bundles, and hence from
Proposition \ref{prop-n} it follows that $\scrC(f)$ contains only \'etale trivial bundles. Therefore,
$\scrC_P(f)$ contains only \'etale trivial bundles. In other words, $\scrC_P(f)$ is a full subcategory of
the neutral Tannakian category $\Vect^{et}(X,\,P)$. Consequently, there is
a natural surjection between their Tannaka duals
$$
\pi_1^{et}(X,\,P)\, \longrightarrow\, A\, \longrightarrow\, e ,
$$
where $A$ is as in \eqref{alpha}.
Hence the induced $A$-cover $(Y',\,Q)\, \longrightarrow\, (X,\,P)$ is \'etale, where $Q$ is defined above.

Let $\psi\,:\,(Z,\,P')\, \longrightarrow\, (X,\,P)$ be any \'etale cover which is dominated by $X_e$. Let
$\scrC(\psi)$ denote the neutral Tannakian subcategory of $\Vect^{ss}(X,\,B_f)$ defined by the full subcategory 
generated by $\psi_*{\mathcal O}_Z$.
Then $\scrC(\psi)$ is a subcategory of $\scrC_P(f)$. Hence $Z$ is dominated by $Y'$ establishing
that $\phi: (Y',\,\phi^*P)\,\longrightarrow\, (X,\,P)$ is the maximal \'etale cover of $(X,\,P)$ dominated by $X_e$.
\end{proof}

\begin{proposition}\label{main}
Let $f:Y\, \longrightarrow\, X$ be a finite generically smooth morphism, and let $P$ be a geometric branch data on $X$. Let
$\scrC_P(f)$ be the full Tannakian subcategory of $\scrC(f)$ 
consisting of objects which are from $\Vect(X,\,P)$ (in the sense of \eqref{2.2}). Let $A$ be the Tannaka dual of $\scrC_P(f)$, and
let $\alpha\,:\,G\,\longrightarrow\, A$ be the natural epimorphism. Then the following five statements are equivalent:
\begin{enumerate}
\item The morphism $(Y,\,f^*P)\,\longrightarrow\, (X,\,P)$ is genuinely ramified.

\item $\alpha(\Gal(k(X_e)/Y))\,=\,A$.

\item $\pi_1(f)\,:\,\pi_1^{S}(Y,\,f^*P)\,\longrightarrow\, \pi_1^{S}(X,\,P)$ is surjective.

\item $\pi_1(f)\,:\,\pi_1^{N}(Y,\,f^*P)\,\longrightarrow\, \pi_1^{N}(X,\,P)$ is surjective.

\item $\pi_1(f)\,:\,\pi_1^{et}(Y,\,f^*P)\,\longrightarrow\, \pi_1^{et}(X,\,P)$ is surjective.
\end{enumerate}
\end{proposition}

\begin{proof}
 The equivalence of (1) and (5) is trivial. Also note that (3) implies (4) and (4) implies (5).
 
Let $H'\,=\,{\rm kernel}(\alpha)$, and let $\phi\,:\, Y'\,\longrightarrow\, X$ be the normalization of $X$ in $k(X_e)^{H'}$. By
Lemma \ref{lem-n}, $$\phi\,:\, (Y',\,\phi^*P)\,\longrightarrow\, (X,\,P)$$ is the maximal \'etale cover of $(X,\,P)$ dominated by $X_e$.
 Hence (1) is equivalent to the statement that $k(Y')\cap k(Y)\,=\,k(X)$. But this is equivalent to the subgroup $H'
\, \subset\, \Gal(k(X_e)/k(Y))$ being the whole group $G$, which in turn is equivalent to (2).

To prove that (5) implies (3) we need to show that
\begin{enumerate}
\item[{(i)}] the functor $f^*:\Vect^{ss}(X,\,P)\,\longrightarrow\, \Vect^{ss}(Y,\,f^*P)$ is fully faithful, and 

\item[{(ii)}] subobjects of $f^*\scrE$ are pullback bundles.
\end{enumerate}
(See \cite[p.~139, Proposition 2.21(a)]{deligne-milne}.)

Since $f\,:\,(Y,\,f^*P)\,\longrightarrow\, (X,\,P)$ is genuinely ramified, the map $g\,:\,Z\,\longrightarrow\, W$ in \eqref{d1} is genuinely 
ramified by Lemma \ref{4-main}. Hence the natural homomorphism
$$H^0(W,\,\,Hom(V_1,\,V_2))\,\,\longrightarrow\,\, H^0(Z,\,\,Hom(g^*V_1,\,g^*V_2))$$ is 
an isomorphism by \cite[Lemma 4.3]{BP-GR}. When $V_1$ and $V_2$ are $\Gamma$-equivariant, the above natural map is $\Gamma$-equivariant; 
consequently, we have
$$H^0(W,\,Hom(V_1,\,V_2))^{\Gamma} \,\cong\, H^0(Z,\,Hom(g^*V_1,\,g^*V_2))^{\Gamma}.$$

We will now show that all subobjects of $f^*\scrE$ are of the form $f^*{\mathcal V}$, where ${\mathcal V}\, \subset\, \scrE$ is
a subobject. Let $E$ be the $\Gamma$--bundle on $W$ representing $\scrE$. First assume that $\scrE$ is stable. Then $f^*\scrE$ is stable by Theorem \ref{stablepullback}. Therefore,
any subobject of $f^*\scrE$ is either $f^*\scrE$ or $0$. Hence all subobjects of $f^*\scrE$
are of the form $f^*{\mathcal V}$, where ${\mathcal V}\, \subset\, \scrE$. Next assume that $\scrE$ is polystable. So
$$
E\,=\, \bigoplus_{i=1}^n {E}_i\otimes T_i\, ,
$$
where ${E}_1, \, \cdots,\, {E}_n$ are stable $\Gamma$--bundles such that ${E}_i\, \not=\, {E}_j$ if $i\, \not=\, j$,
and $T_1,\, \cdots,\, T_n$ are trivial $\Gamma$--bundles on $W$. In fact $T_i$ is the trivial $\Gamma$--bundle on $W$ with fiber
$$H^0(W, \, \text{Hom}({E}_i,\, E))^{\Gamma}.$$ Let $r_i$ be the rank of $T_i$. Consider the pullback
$$
g^*E\,=\, \bigoplus_{i=1}^n g^*{E}_i\otimes \cO_{Z}^{r_i}.
$$
From Theorem \ref{stablepullback} it follows that each $g^*{E}_i$ is a stable $\Gamma$--bundle. For $1\, \leq\, i,\, j\, \leq\, n$, we
again have (by \cite[Lemma 4.3]{BP-GR}),
$$
H^0(W,\, \text{Hom}({E}_i,\, {E}_j))\,\cong \, H^0(Z,\, \text{Hom}(g^*{E}_i,\, g^*{E}_j)),
$$
and hence $$H^0(Z,\, \text{Hom}(g^*{E}_i,\, g^*{E}_j))^{\Gamma}\,=\, H^0(W,\, \text{Hom}({E}_i,\, {E}_j))^{\Gamma}\,=\, 0,$$ because
${E}_1, \, \cdots,\, {E}_n$ are pairwise non-isomorphic stable $\Gamma$--bundles. Hence any subobject of
$g^*E$ is of the form $\bigoplus_{i=1}^n g^*{E}_i\otimes T'_i$, where $T'_i\, \subset\, \cO_{Z}^{r_i}$ is a trivial $\Gamma$-subbundle.
Consequently, all subobjects of $f^*\scrE$
are of the form $f^*{\mathcal V}$, where ${\mathcal V}\, \subset\, \scrE$.

Finally, for a general subobject of $\scrE$, let
\begin{equation}\label{fw}
0\, \subset\, {F}_1\, \subset\, \cdots \, \subset\, {F}_{\ell-1}
\, \subset\, {F}_{\ell} \,=\, E
\end{equation}
be the Jordan--H\"older filtration of the $\Gamma$--bundle $E$; so for any $1\, \leq\, i\, \leq\, \ell$, the quotient bundle
${F}_i/{F}_{i-1}$ is the unique maximal polystable subbundle of $E/{F}_{i-1}$
\cite[p.~24, Lemma 1.5.5]{HL}. This uniqueness ensures that each $F_i$ is preserved by the action of $\Gamma$ on $E$.
Let
$$
V\, \subset\, g^*E
$$
be a semistable $\Gamma$-subbundle over $Z$ of degree zero. Let
$$
{V'}\, \subset\, g_*{V}
$$
be the maximal semistable subsheaf (the first nonzero term in the Harder--Narasimhan filtration). Then using the above observations
if follows that $\text{degree}(V')\,=\, 0$ and $\text{rank}(V')\,=\, \text{rank}(V)$ (see \cite{BP2}). Moreover, the natural
homomorphism $g^*g_*V\, \longrightarrow\, V$ has the property that its restriction to $g^*V'$ is an isomorphism. This
completes the proof.
\end{proof}

Note that the category $\Vect(X,\,O)$ is the same as $\Vect(X)$, the category of vector bundles on $X$.

Proposition \ref{main} has the following immediate consequence:

\begin{corollary}
 Let $\scrC_0(f)$ be the full Tannakian subcategory of $\scrC(f)$ consisting of objects from $\Vect(X)$. Then $\scrC_0(f)$ is
a Tannakian category. Let $A$ be its Tannaka dual and $a\,:\,G\,\longrightarrow \, A$ the natural epimorphism. Let $H\,=\,
\Gal(k(X_e)/k(Y))$ and let $B$ be the image of $H$ in $A$. Then the following are equivalent:
\begin{enumerate}
\item $f\,:\,Y\,\longrightarrow \, X$ is genuinely ramified.

\item $B\,=\,A$.
\end{enumerate}
\end{corollary}

\section{Pushforward of the structure sheaf}

Consider the diagram in \eqref{cd}.
Since $\cO_{X_e}\,=\,b^*\cO_X$, it follows that $\cO_{X_e}$ has a natural $G$-equivariant structure. Note that $\cO_{X_e}$ is a
subbundle of $E$ (defined in \eqref{de}) preserved by the action of
$G$. Since $\cO_{X_e}$ is in $\Vect(X,\,O)$, by Lemma \ref{semistability} it is also in $\Vect(X,\,P\cap B_f)$
and its degree is $0$. So
by Lemma \ref{maximalsubbundle} there exists a unique maximal semistable $G$-equivariant subbundle
\begin{equation}\label{ef}
F\, \subset\, E
\end{equation}
of degree $0$ such that $F\,\in\, \Vect(X,\,P\cap B_f)$.

Let $U$ be the normalized fiber product of $b\,:\,X_e\,\longrightarrow\, X$ and $\alpha\,:\,W\,\longrightarrow\, X$. Let $a\,:\,U\,
\longrightarrow\, W$ and $\beta\,:\, U\,\longrightarrow\, X_e$ be the natural projections. So we have the following diagram:
\begin{equation}\label{e-a}
\xymatrix{
& & & Y_e\ar[dd]^G\ar[ld]_{\scrf} \\
& U\ar[r]^{\beta}\ar[dd]_a & X_e\ar@{-->}[rd]\ar[dd]_b^G \\
& & & Y\ar[ld]_f\\
& W \ar[r]^{\alpha} & X
}
\end{equation}
\begin{lemma}\label{lem5.1}
Let $F_W$ be the maximal degree zero $\Gamma$-subbundle of $(a_*\beta^*E)^G$. The $\Gamma\times G$-equivariant bundles $\beta^*F$
(see \eqref{ef}) and $a^*F_W$ are isomorphic.
\end{lemma}

\begin{proof}
Note that $F_W$ is a slope zero subbundle of $a_*\beta^*E$. Since $F$ is the maximal subbundle of $E$ of slope zero, and $E/F$ is
a semistable bundle of negative slope, the image of $F_W$ in $a_*\beta^*E$ lies in $a_*\beta^*F$. By adjointness we get a natural
map of bundles
\begin{equation}\label{ef2}
a^*F_W\,\longrightarrow \, \beta^*F
\end{equation}
over $U$.

It can be shown that away from the preimage of $B_f\cup P$ the homomorphism 
in \eqref{ef2} is an isomorphism. Indeed, this follows immediately from the following two facts:
\begin{enumerate}
\item For any generically smooth surjective map $\varphi\, :\, Z_1\, \longrightarrow\, Z_2$ of smooth projective curves, and any
vector bundle $E$ on $Z_2$, the Harder--Narasimhan filtration of $\varphi^*E$ is the pullback of the Harder--Narasimhan filtration of
$E$ (see \cite[p.~12823--12824, Remark 2.1]{BP-GR}).

\item If $\varphi\, :\, Z_1\, \longrightarrow\, Z_2$ is an \'etale Galois covering map of smooth curves with Galois group
$G$, and $E$ is a $G$-equivariant vector bundle on $Z_1$, then the natural map $$\varphi^*((\varphi_* E)^G)\,\longrightarrow \, E$$
is an isomorphism.
\end{enumerate}

Since both $a^*F_W$ and $\beta^*F$ are degree zero bundles, the generically isomorphic homomorphism
in \eqref{ef2} is actually an isomorphism. Indeed, the cokernel of this map is a torsion sheaf and its degree
is $\text{degree}(\beta^*F)- \text{degree}(a^*F_W)\,=\, 0$, and hence the cokernel is the zero sheaf.
Consequently, the homomorphism in \eqref{ef2} is an isomorphism.
\end{proof}

The goal is to prove the following:

\begin{theorem}\label{5.2}
Let $f\,:\,(Y,\,f^*P)\,\longrightarrow\, (X,\,P)$ be a morphism of formal orbifolds. Assume
$f$ is also a Galois cover. Then the following are equivalent:
\begin{enumerate}
\item $f\,:\,(Y,\,f^*P)\,\longrightarrow\, (X,\,P)$ is genuinely ramified.

\item $F\,=\,\cO_{X_e}$.

\item $F_W\,=\,\cO_W$.
\end{enumerate}
\end{theorem}

\begin{proof}
The equivalence of (2) and (3) is a consequence of Lemma \ref{lem5.1}.
 
Since $f$ is Galois $Y\,=\,X_e$ and $E=\cO_{X_e}\times k[G]$. Note that
$$a_*\beta^*E\,=\,a_*\beta^*\cO_{X_e}\otimes_k k[G]\,=\,a_*\cO_U\otimes_k k[G]$$
(see \eqref{e-a}). Hence $[a_*\beta^*E]^G\,=\,a_*\cO_U$. 
 
Since $f$ is genuinely ramified by Lemma \ref{4-main}, the map $a\,:\, U\,\longrightarrow \, W$ is genuinely ramified. Therefore the degree
zero part of the Harder--Narasimhan filtration of $a_*\cO_U$ is actually
$\cO_W$ (by \cite{BP-GR}). So $F_W$, which is the degree zero
part of the Harder--Narasimhan filtration of $[a_*\beta^*E]^G$, is also $\cO_W$. This proves that (1) implies (3).

Suppose that $f\,:\,(Y,\,f^*P)\,\longrightarrow\, (X,\,P)$ is not genuinely ramified.
Let $$g\,:\,Y'\,\longrightarrow\, X$$ be the maximal 
intermediate cover such that $g\,:\,(Y',\,g^*P)\,\longrightarrow \, (X,\,P)$ is \'etale. Let $Z'$ be the normalized fiber product of $Y'$ and 
$W$ over $X$, and let $h\,:\,Z'\,\longrightarrow\, W$ be the natural projection. Then $h$ is \'etale and it is dominated by $a$. Hence 
$h_*\cO_{Z'}$ is a subbundle of $a_*\cO_U$ of degree zero, and its rank is the same as the degree of $h$. But $F_W$ is the maximal degree zero
subsheaf of $[a_*\beta^*E]^G\,=\,a_*\cO_U$. This contradicts (3). Hence (3) implies (1).
\end{proof}

\section*{Acknowledgements}

We are very grateful to the referee for helpful comments to improve the manuscript.
The first author is partially supported by a J. C. Bose Fellowship (JBR/2023/000003).

\end{document}